\documentclass{article}

\usepackage{amsmath}
\usepackage{amsthm}
\usepackage{amssymb}
\usepackage{delarray}

\usepackage[dvips]{graphics}
\usepackage{epsfig}
\usepackage{color}
\usepackage[all]{xy}

\newtheorem{thm}{Theorem}[section]
\newtheorem{lemma}[thm]{Lemma}
\newtheorem{prop}[thm]{Proposition}
\newtheorem*{cor}{Corollary}

\theoremstyle{definition}
\newtheorem*{defn}{Definition}

\theoremstyle{remark}
\newtheorem*{remark}{Remark}

\theoremstyle{definition}
\newtheorem*{ex}{Example}

\title{Reduced decompositions and commutation classes}

\author{Delong Meng}

\date{c/o Department of Mathematics\\ Massachusetts Institute of Technology\\ Cambridge, MA 02139-4307\\ \ \\ Email: \texttt{delong13@mit.edu}}

\begin{document}

\maketitle

\begin{abstract}
	We study three aspects of commutation classes of reduced decompositions: the number of commutation classes, the structures of their corresponding graphs, and the enumeration of subnetworks, a concept recently introduced by Warrington \cite{Warrington}. Our bound for the number of commutation classes generalizes the works of Knuth\cite{Knuth}, Green and Losonczy \cite{Green}, and Tenner \cite{Tenner}. We analyze the structure of the graph $G(w)$ using pattern avoidance, which provides an application of Tenner's characterization of vexillary permutations in \cite{Tenner}. We also discuss some connections between our work and recent developments in the strong Bruhat order and the higher Bruhat order.
\end{abstract}

\section{Introduction}

Given $w\in S_n$, let $l(w)=\{(i, j): i<j, w(i)>w(j)\}$ denote the number of inversions of $w$. A $\textit{reduced decomposition}$ is a sequence $(i_1, i_2, \ldots, i_{l(w)})$ such that $w=s_{i_1}s_{i_2}\cdots s_{i_{l(w)}}$, where $s_i$ denotes the transposition of $w(i)$ and $w(i+1)$. Two reduced decompositions are said to be in the same \textit{commutation class} if they can be obtained from each other by applying the relation $s_is_j=s_js_i$ for $|i-j|\ge 2$.

For example, (1, 2, 1, 3, 2) is a reduced decomposition of 3421 because 3421 can be obtained from 1234 as follow:
$$\textbf{12}34\to 2\textbf{13}4\to \textbf{23}14\to 32\textbf{14}\to 3\textbf{24}1\to 3421.$$

The reduced decomposition (1, 2, \textbf{3}, \textbf{1}, 2) is in the same commutation class as (1, 2, 1, 3, 2) because $s_1s_3=s_3s_1$.\\

The set of commutation classes has received considerable attention in a variety of contexts. Algebraically, it encodes the structures of the Weyl groups \cite{Bedard}. Geometrically, it counts the number of pseudoline arrangements \cite{Felsner}. Combinatorially, it brings pattern avoidance to the study of reduced decompositions \cite{Tenner}.

Despite its importance, much remains to be understood about the set of commutation classes. For example, the number of commutation classes is known only for few special cases of $w$.

In November 2009, Warrington \cite{Warrington} introduced a new combinatorial object called \textit{subnetworks}. The concept of subnetworks, together with the author's work on the higher Bruhat order \cite{Meng}, shed new light on this subject. In this paper, we study three aspects of commutation classes:

\begin{itemize}
	\item How many commutation class are there? (Section 3)
	\item What are the relationships between them? (Section 4)
	\item Enumeration of subnetworks. (Section 5)
\end{itemize}

In Section 2, we give some definitions and background information. In particular, we define a graph $G(w)$ whose vertices correspond to commutation classes of reduced decompositions of $w$. This graph is the main object of our study.

In Section 3, we give a bound for the number of vertices of $G(w)$. Most of the previous research on this topic focused entirely on the permutation $w_0=n, n-1, \ldots, 1$ (see Knuth \cite{Knuth} and Felsner \& Valtr \cite{Felsner}). We extend their results to arbitrary $w$, which encapsulates all previously known bounds in \cite{Green, Knuth, Tenner}. The technique in the proof of the upper bound is inspired by the author's recent work \cite{Meng} on the higher Bruhat order.

In Section 4, we analyze the structure of $G(w)$ for some special $w$ using pattern avoidance. We generalize Green and Losonczy's freely braided permutations in \cite{Green} to rectangular permutations. Our method is inspired by the author's recent work \cite{Meng1} on the Boolean elements of twisted involutions. Our result provides an application of Tenner's characterization of vexillary permutations in \cite{Tenner}.

Section 5 focuses on some enumerations of subnetworks, which in some sense are permutation patterns for commutation classes.

Even though Sections 3, 4, 5 study three different topics, they are intrinsically connected. For example, to prove the lower bound of the size of $G(w)$, we use subnetworks to define a partial order on the set of commutation classes, which helps us to analyze the structure of $G(w)$.

Since our study draws many ideas from recent papers, there is great potential for further development. For example, the enumeration of subnetworks is a new field of study. We discuss some possible directions for future research in Section 6.

\section{Definitions and Background}

	Refer to the first paragraph of Section 1 for the definitions of reduced decompositions and commutation classes.
	
	It is well-known that any two reduced decompositions can be obtained from each other by applying the relations $s_is_j=s_js_i$ for $|i-j|\ge 2$ and $s_is_{i+1}s_i=s_{i+1}s_is_{i+1}$. (See \cite{Stanley}.) The former is called a \textit{short braid move} and the latter a \textit{long braid move}.
	
	It is convenient to represent commutation classes as vertices of a graph whose edges correspond to long braid moves.
	
	\begin{defn}
		Given $w\in S_n$, define the graph $G(w)$ as follow:
		\begin{itemize}
			\item Each vertex $v$ of $G(w)$ corresponds to a commutation class $g(v)$ of reduced decompositions of $w$.
			\item Two vertices of $v_1$ and $v_2$ share an edge if there exists a long braid move between two reduced decompositions $\rho_1$ and $\rho_2$, where $\rho_1\in g(v_1)$ and $\rho_2\in g(v_2)$.
		\end{itemize}
	\end{defn}
	
	Elnitsky \cite{Elnitsky} showed that $G(w)$ is connected and bipartite.
	
	\begin{ex}
		Below is the graph $G(3421)$. Note that $G$ is not a directed graph, but for convenience, we use arrows to indicate the long braid move $(i, i-1, i)\to (i-1, i, i-1)$.
		
		\xymatrix{
			&(2, 1, 3, 2, 3), (2, 3, 1, 2, 3) \ar[d]&\\
			&(2, 1, 2, 3, 2) \ar[d]&\\
			&(1, 2, 1, 3, 2), (1, 2, 3, 1, 2) &
		}
	\end{ex}
	
	Note that each long braid move changes the sum of indices by exactly 1. (The sum of indices of $(i_1, i_2, \ldots, i_l)$ is $i_1+i_2+\cdots+i_l$.) Thus it is natural for us to define a poset on the commutation classes ranked by the sum of indices. We call a long braid move an \textit{upward move} if it increases the sum of indices by 1, and a \textit{downward move} otherwise.
	
	\begin{defn}
		Let $P(w)$ denote the partial order on the commutation classes of reduced decompositions of $w$ whose cover relation is defined as follow. A commutation class $A$ covers $B$ if there exist reduced decompositions $\rho_a\in A$ and $\rho_b\in B$ such that $\rho_b$ is obtained from $\rho_a$ by a downward move.
	\end{defn}

	\begin{remark}
		The poset $P(w_0)$ is the higher Bruhat order $B(n, 2)$, where $w_0=n, n-1, \ldots, 1$. (See \cite{Ziegler}.)
	\end{remark}

	We frequently use pattern avoidance in our study of commutation classes.
	
	\begin{defn}
		Let $w=w(1)w(2)\cdots w(n)$ and $p=p(1)p(2)\cdots p(k)$. The permutation $w$ contains a $p$-pattern if there exist $i_1<\cdots<i_k$ such that $w(i_i)\cdots w(i_k)$ is in the same relative order as $p(1)\cdots p(k)$. That is, $w(i_h)<w(i_j)$ if and only if $p(h)<p(j)$. Furthermore, let $N_p(w)$ denote the number of $p$-patterns in $w$. If $N_p(w)=0$, then $w$ is \textit{p-avoiding}.
	\end{defn}
	
	We now define subnetworks introduced by Warrington \cite{Warrington}.
	
	\begin{defn}
    Let $X$ be a set of reduced decompositions of an element $p$ of $S_m$. Given a reduced decomposition $\rho$ of a permutation $w\in S_n$, we pick $m$ distinct integers from $\{1, 2, \ldots, n\}$. Considering their relative positions, $\rho$ induces a reduced decomposition of these $m$ integers. An \textit{$X$-subnetwork of $\rho$} is a set of $m$ distinct integers such that their reduced decomposition (induced by $\rho$) is in $X$.
	\end{defn}

	\begin{ex}
		Given $a<b<c<d$, we keep track of the relative positions of this quadruple.
\[
\begin{array}{llll}
 & (\dots, d, \dots, c, \dots, b, \dots, a, \dots) & & \\
\to \cdots \to & (\dots, d, \dots, b, \dots, c, \dots, a, \dots) & \to \cdots \to &
(\dots, d, \dots, b, \dots, a, \dots, c, \dots) \\
\to \cdots \to & (\dots, d, \dots, a, \dots, b, \dots, c, \dots) & \to \cdots \to &
(\dots, a, \dots, d, \dots, b, \dots, c, \dots) \\
\to \cdots \to & (\dots, a, \dots, b, \dots, d, \dots, c, \dots) & \to \cdots \to &
(\dots, a, \dots, b, \dots, c, \dots, d, \dots)
\end{array}
\]
	This reduced decomposition of $w$ induces the reduced decomposition $(2,3,2,1,2,3)$ of $(a, b, c, d)$, and thus $(a, b, c, d)$ is a $(2,3,2,1,2,3)$-subnetwork.
	\end{ex}


\section{The size of $G(w)$}

This section is dedicated to the proof of the following two theorems.


	\begin{thm}
		\label{bounding}
		Given $w\in S_n$, let $Y$ denote the maximum number of long braid moves a reduced decomposition of $w$ has, then
		$$2^{\lceil \frac Y 2\rceil}+N_{321}(w)- {\left\lceil \frac Y 2\right\rceil} \le |G(w)|<3^{l(w)},$$
		where $|G(w)|$ denote the number of vertices of $G(w)$.
	\end{thm}
	
	\begin{remark}
		For $n$ large enough, the upper bound can be reduced to $2.487^{l(w)}$.
	\end{remark}

	
	\begin{thm}\label{boundSum}
		Given $w_1, w_2, \ldots, w_k\in S_n$ such that $l(w_1)=l(w_2)=\cdots=l(w_k)=l$, then
		$$|G(w_1)|+|G(w_2)|+\cdots+|G(w_k)|<4^{l+n}.$$
	\end{thm}
	
	
	\begin{proof} [Proof of Theorem \ref{boundSum}]
		Define an \textit{inversion} of a reduced decomposition as a pair of indices $(i, j)$ such that $i>j$ and $i$ is to the left of $j$. For each commutation class, we call a reduced decomposition a \textit{representative element} if for any consecutive $i, j$ with $i-j\ge 2$, $i$ is to the left of $j$. Such a representative exists because we can pick an element with the most of number of inversions. Suppose this representative element is $i_1 i_2\cdots i_l$.
		
		Since $w_1, w_2, \ldots, w_k$ all have $l$ inversions, there is an injective map between their commutation classes and the representative elements.
		
		We now apply the same trick as Knuth $\cite{Knuth}$ (p. 36). We first write down $i_1$ left parentheses. Thereafter, we write down $i_k-i_{k+1}+1$ right parentheses followed by one left parenthesis for each $1\le k \le l$, and we finish with $i_l$ right parentheses. This process yields a balanced string of $l+i_1-1$ pairs of parentheses. The number of such strings is fewer than the Catalan number $C_{l+n-1}=\dfrac{1}{l+n}\dbinom{2(l+n-1)}{l+n-1}$, which is less than $4^{(l+n)}$.
	\end{proof}
	
	The rest of this section contains the proof of Theorem \ref{bounding}.	We prove the lower bound by construction, and we prove the upper bound by extending Knuth's ``cutpath" technique \cite{Knuth} to arbitrary $w$.


	\subsection{Lower bound}
	
	We first prove that there is a hypercube of dimension $\lceil \frac Y 2\rceil$ embedded in $G(w)$, which has $2^{\lceil \frac Y 2\rceil}$ vertices. Then we construct $N_{321}(w)-\lceil \frac Y 2\rceil$ more by analyzing the poset $P(w)$.
	
	\subsubsection{Cube construction} \label{cubeConstr}
	
	\begin{prop} \label{cube}
		There exists a hypercube of dimension $\lceil \frac Y 2\rceil$ embedded in $G(w)$.
	\end{prop}
	
	\begin{proof}
		Let $\rho$ denote a reduced decomposition with $Y$ long braid moves. Without loss of generality, assume that at least $\lceil \frac Y 2\rceil$ of these moves are downward moves. We claim that these moves do not intersect with each other.
		
		Suppose on the contrary that there exist two intersecting downward moves, then $\rho$ must contain a substring of the form $i, i-1, i, i-1, i$, which cannot appear in any reduced decompositions. Therefore no downward moves can intersect with each other.
		
		Now we number the downward moves $1, 2, \ldots, k$, where $k\ge Y/2$. Since these moves are independent of each other, applying any subset of these moves changes $\rho$ to a reduced decomposition in a different commutation class as illustrated in the picture below.\\
		
		\xymatrix{
			& w=42615378 &\\
			& (3, \underline{2, 1, 2}, \underline{5, 4, 5}, 3, \underline{7, 6, 7})\ar[dl]\ar[d]\ar[dr] &\\
			(3, \underline{1, 2, 1}, \underline{5, 4, 5}, 3, \underline{7, 6, 7}) \ar[d]\ar[dr] & (3, \underline{2, 1, 2}, \underline{4, 5, 4}, 3, \underline{7, 6, 7})\ar[dl]\ar[dr] & (3, \underline{2, 1, 2}, \underline{5, 4, 5}, 3, \underline{6, 7, 6}) \ar[dl]\ar[d]\\
			(3, \underline{1, 2, 1}, \underline{4, 5, 4}, 3, \underline{7, 6, 7})\ar[dr] & (3, \underline{1, 2, 1}, \underline{5, 4, 5}, 3, \underline{6, 7, 6}) \ar[d] & (3, \underline{2, 1, 2}, \underline{4, 5, 4}, 3, \underline{6, 7, 6})\ar[dl]\\
			&(3, \underline{1, 2, 1}, \underline{4, 5, 4}, 3, \underline{6, 7, 6})&
		}

		\ \newline Thus, these $k$ downward moves generate a hypercube of dimension $k$. Since $k\ge Y/2$, there exists a hypercube of dimension $\lceil \frac Y 2\rceil$ embedded in $G(w)$.
	\end{proof}
	
	One consequence of Proposition \ref{cube} is that if all upward and downward moves do not intersect with each other (which is called \textit{freely braided} in \cite{Green}), then entire graph $G(w)$ is a hypercube.
	
	\begin{cor} (\cite{Green} Theorem 5.2.1)
		\label{freeBraid}
    If $w$ is freely braided, then there are exactly $2^Y$ commutation classes.
	\end{cor}

	

	\subsubsection{Poset of commutation classes} \label{poset}
	
	Consider the poset $P(w)$ on commutation classes ranked by the sum of indices. Since the cube we constructed in Section \ref{cubeConstr} occupies exactly $\lceil \frac Y 2\rceil+1$ ranks of $P(w)$. We claim that the $P(w)$ has $N_{321}(w)+1$ ranks, which would imply that there are at least $N_{321}(w)-\lceil \frac Y 2\rceil$ vertices not lying on the cube.
	
		\begin{lemma} \label{rankLemma}
     The poset $P(w)$ has $N_{321}(w)+1$ ranks.
		\end{lemma}

		\begin{proof}
    	We consider the number of $212$-subnetworks. (For simplicity, we use 212 instead of (2, 1, 2).) It's easy to see that all reduced decompositions in a commutation class have the same number of $212$-subnetworks. Moreover, changing $i+1, i, i+1$ to $i, i+1, i$ decreases the number of $212$-subnetworks by exactly $1$. Since the number of $212$-subnetworks is at least $0$ and at most $N_{321}(w)$, we only need to show that both extremes are achievable.

    Starting from $1, 2, \ldots, n$, the following two reduced decompositions would do the job:

    	\begin{itemize}
        \item \label{zeroCase} Move $w(n)$ to the $n$th position, then move $w(n-1)$ to the $(n-1)$th position, and so on. We basically move every number from $n$ to $1$ to its final position. This gives $0$ subnetworks of $212$. (For example, consider the reduced decomposition $123121$ of the permutation $4321$.)
        \item Move $w(1)$ to the first position, then move $w(2)$ to the second position, and so on. In this way, any triple that forms a $321$-pattern would yield a $212$ subnetwork.
    	\end{itemize}
		\end{proof}

\noindent This finishes the proof for the lower bound. We end this section with a class of permutations characterized in $\cite{Tenner}$.

		\begin{cor} \label{pathGraph}
    	If $G(w)$ is a line, then $|G(w)|=N_{321}(w)+1$.
		\end{cor}



	\subsection{Upper bound} \label{upperbound}
	
 	We prove the upper bound by counting the number of tilings of an Elnitsky Polygon.
		
		\begin{defn}
			Given $w\in S_n$, an \textit{Elnitsky Polygon $E(w)$} is a polygon with $2n$ sides (all with unit length) labeled $1, 2, \ldots, n, w(1), w(2), \ldots, w(n)$ such that the sides labeled $i$ and $w(i)$ are parallel.
		\end{defn}
		
		Elnitsky \cite{Elnitsky} showed that there is a bijection between commutation classes of reduced decompositions of $w$ and the tilings of $E(w)$ with rhombi of unit sides.
		
		\begin{ex}
			Below is an example of rhombus tiling of $E(543162)$.
			\begin{center}
				\includegraphics[width=0.50\textwidth]{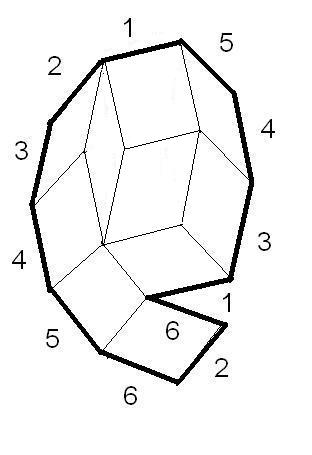}
			\end{center}
		\end{ex}
		
		We prove the follow equivalent form of the upper bound of Theorem \ref{bounding}.
		
		\begin{prop}\label{boundTiling}
			There are at most $3^A$ tilings of an Elnitsky polygon of area $A$. Furthermore, 3 can be replace by 2.487 for sufficiently large $A$.
		\end{prop}
		
		\begin{proof}
			We induct on $A$. When $A=3$, there are 2 tilings, which is fewer than $3^3$. (In fact, the bound is weak for small $A$.)
			
			The key idea of the inductive step is to delete the strip that connects the two sides labeled 1. Suppose this polygon is $E(w)$ where $w(i)=1$. Then this strip has length $i$. After deleting this strip, we obtain a polygon $E(w')$, where $w'=w(1)w(2)\cdots w(i-1) w(i+1)\cdots w(n)$. The picture below illustrates the main idea of our proof.
			
			\begin{center}
				\includegraphics[width=0.80\textwidth]{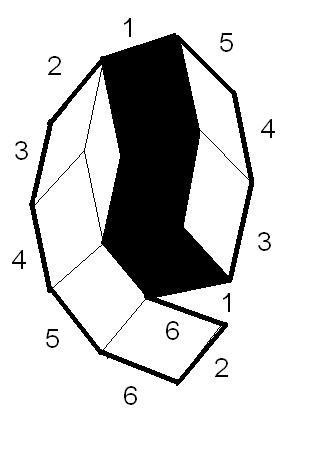}
			\end{center}
			
			By the inductive hypothesis, there are at most $3^{A-i}$ rhombus tilings of $E(w')$. Thus, it remains to show that there are at most $3^i$ strips that connects the two sides labeled 1. This was done in Knuth \cite{Knuth} p. 39.
			
			In a recent paper \cite{Felsner}, Felsner and Valtr reduced the bound $3^i$ to $4i\cdot 2.486976^i$. Thus for $A$ sufficiently large, the number of tilings is at most $2.487^A$.
		\end{proof}
		
		\begin{remark}
			Our result has an analogue for Coxeter groups of type B. The number of commutation classes is at most $3^{r(w)}$, where $r(w)$ is the rank of $w$ in the Bruhat order on Coxeter groups of type B. We omit the details here. Interested readers are referred to \cite{Bjorner} for the background reading of Coxeter groups and \cite{Elnitsky} for an analogue of Elnitsky Polygon for type B.
		\end{remark}



\section{The structure of $G(w)$}

In this section, we first study the cycles of $G(w)$. Then we classify permutations $w$ such that $G(w)$ is a higher dimensional rectangle, which is a generalization of Green and Losonczy's freely braided permutations in \cite{Green}. Finally, we study some other permutations for which $G(w)$ has a nice geometric structure.

	\subsection{Cycles of $G(w)$}
	
		We first define induced cycles.
		
		\begin{defn}
			An \textit{induced cycle} in the graph $G$ is a cycle with $k$ vertices $a_1, a_2, \ldots, a_k$ such that $a_i$ is connected to $a_{i+1}$ (with indices mod $k$) and no edges exist between $a_i$ and $a_j$ for $|i-j|\neq 1$.
		\end{defn}
		
		Shapiro et al. proved that $G(w)$ is generated by 4-cycles and 8-cycles in $\cite{Shapiro}$. We now characterize the conditions for which two edges can lie on an induced cycle.


		\begin{lemma} \label{cycles}
		Given two edges of $G(w)$ drawn from the same vertex, we have the following three cases:
		
			\begin{itemize}
				\item They lie on an induced 4-cycle if the two long braid moves are applied to 6 indices.
				\item They lie on an induced 8-cycle if both edges are applied to a substring that has the form of a reduced decomposition of the longest word in $S_4$ (for instance $(i+1, i, i+1, i-1, i, i+1)$).
				\item They do not lie on an induced cycle.
			\end{itemize}
			
		\end{lemma}

		\begin{proof}
			
			First, we clarify that an edge between two vertices must be the set of long braid moves acting on the same three indices. (In other words, when we say the edge is applied to $i+1, i, i+1$, we mean that all the long braid moves act on these indices.)

			Notice that two long braid moves either act on 5 or 6 indices. If they act on 6 indices, then they lie on an induced 4-cycle. Suppose they act on 5 indices, and these 5 indices belong to a substring that has the form of a reduced decomposition of the longest word in $S_4$. Since the graph of the longest word of $S_4$ is an induced 8-cycle. These two long braid moves lie on an induced 8-cycle.

			We are now left to show that in all other cases these two edges cannot lie on an induced cycle. Since they act on 5 indices, there are three cases to consider:
			
			\begin{enumerate}
  			\item $i+1, i, i+1, i+2, i+1$
  			\item $i, i+1, i, i-1, i$
  			\item $i-1, i+1, i, i+1, i-1$
			\end{enumerate}
			
			Since none of the three can form a reduced decomposition of the longest word in $S_4$, we need a long braid move that does not act on any of these indices to complete a cycle. But then the cycle is not an induced cycle because this new long braid move can be applied to all three vertices.
	
		\end{proof}


	\subsection{Rectangular permutations}
	
		We first give the rigorous definition of a higher dimensional rectangle.


		\begin{defn}
		
			Let $x\in \mathbb{R}^k$ be a lattice point with coordinates $(x_1, x_2, \ldots, x_k)$, where $x_i\in \mathbb{N}$ for all $i$. An \textit{$x$-rectangle} is defined as the set of lattice points $(y_1, y_2, \ldots, y_n)$ where $0\le y_i\le x_i$ for all $i$, and the set of edges between these points. Two points $y=(y_1, y_2, \ldots, y_n)$ and $z=(z_1, z_2, \ldots, z_n)$ in this set are connected if and only if there exist an $i$ such that $|y_i-z_i|=1$ and $y_j-z_j=0$ for all $j\neq i$.
			
		\end{defn}
		\begin{defn}
    	A \textit{rectangular permutation} is a permutation for which $G(w)$ is isomorphic to an $x$-rectangle for some $x\in \mathbb{R}^k$.
		\end{defn}

		The next proposition is Tenner's characterization of vexillary permutations.
	
		\begin{prop} \label{vex}
			(Tenner \cite{Tenner} Theorem 3.8) Suppose a permutation $w$ contains a $p$-pattern. There exists a reduced decomposition of $w$ which contains a reduced decomposition of $p$ as a consecutive substring if and only if $p$ is 2143-avoiding.
		\end{prop}
		
		We now characterize rectangular permutations using pattern avoidance.
		
		\begin{thm} \label{recPerm}
			A permutation is rectangular if and only if it is 4321, 42531, and 53142-avoiding.
		\end{thm}
		\begin{proof}[Proof of necessity]
			
			We say that two edges emanating from the same vertex \textit{commute} if they lie on a 4-cycle. If $w$ is rectangular, then for any edge $e$ of $G(w)$, there is at most one edge that does not commute with $e$. Suppose on the contrary that $w$ contains a 4321, 42531, or 53142-pattern. By Proposition \ref{vex}, there is a reduced decomposition of $w$ that contains one of the following substrings:
			
			\begin{enumerate}
				\item $(i+1, i, i+1, i-1, i, i+1)$
				\item $(i+2, i, i+1, i, i+2, i-1, i)$
				\item $(i+2, i-1, i, i-1, i+1, i+2, i+1)$
			\end{enumerate}
			
			In the first case, Lemma \ref{cycles} indicates that there is an induced 8-cycle, which cannot occur in a rectangle. In the second and the third cases, there is always an edge that does not commute with at least two edges as illustrated below.\\
			
			\xymatrix{
			&(i+2, i+1, i, i+1, i+2, i-1, i) \ar[d] &(i, i+1, i+2, i+1, i, i-1, i)\ar[dl]\\
			& (i+2, i, i+1, i, i+2, i-1, i) \ar[d] &\\
			& (i+2, i, i+1, i+2, i-1, i, i-1) &
			}

	\ \newline Therefore, $w$ is rectangular only if $w$ is 4321, 42531, and 53142-avoiding.			
		\end{proof}

		
		\begin{proof}[Proof of sufficiency]
		
			Suppose $w$ is 4321, 42531, and 53142-avoiding. We analyze the substrings that each permutation pattern forbids.
    	\begin{enumerate}
      	\item \label{one} 4321-avoiding eliminates induced 8-cycles.
      	\item 4231-avoiding eliminates $(i-1, i+1, i, i+1, i-1)$.
      	\item 4312 and 3421-avoiding eliminates $(i+1, i, i+1, i+2, i+1)$ and $(i, i+1, i, i-1, i)$.
    	\end{enumerate}
    	If the substrings in Case 2 and Case 3 do not intersect, then for every edge in $G(w)$, there is at most one other edge that does not commute with it. (The intersection of Case 2 and Case 3 would yield a 42531 or a 53142-pattern.)
			
			For a given $G(w)$, we construct its corresponding $x$-rectangle explicitly.
			
			We label each commutation class with a vector. Let $M$ denote the maximal element of $P(w)$, the poset on commutation classes. Suppose $M$ has $k$ downward moves. We label $M$ as 0, and we label each commutation class covered by $M$ with an element of the standard basis in $\mathbb{R}^k$.
			
			We now label the commutation classes in order of their ranks. Suppose we have labeled the first $i$ rows of the Hasse diagram of $P(w)$. Let $v$ denote an element in the $(i+1)$st row. If there exist $v_1$ in the $i$th row and $v_2$ in the $(i-1)$st row such that $vv_1$ does not commute with $v_1v_2$, then we label $v$ with $2v_1-v_2$. Since $w$ is 4321, 42531, and 53142-avoiding, there is at most one such pair $(v_1, v_2)$. Thus, this label is well defined. If such $v_1$ and $v_2$ do not exist, then we label $v$ with the sum of all labels of commutation classes that cover $v$.
			
			Let $x$ denote the label of the minimal element of $P(w)$. Then $G(w)$ is isomorphic to a $x$-rectangle. Therefore if $w$ is 4321, 42531, and 53142-avoiding, then $w$ is a rectangular permutation.
		
		\end{proof}
		
				
		\begin{ex}
			The permutation 326514 is rectangular because $G(326514)$ is a (1, 2)-rectangle. The picture below shows the label of each commutation class. For simplicity, we use 21254534 instead of (2, 1, 2, 5, 4, 3, 4) to denote a reduced decomposition. The notation 21245434 (0, 1) means the commutation class of 21245434 is labeled with (0, 1).
			
			\xymatrix{
				& & 21254534 (0, 0)\ar[dl]\ar[dr]&\\
				& 21245434 (0, 1) \ar[dl]\ar[dr] & & 12154534 (1, 0)\ar[dl] \\
				21245343 (0, 2) \ar[dr] & & 12145434 (1, 1)\ar[dl] & \\
				& 12145343 (1, 2) & & \\
			}
		\end{ex}



	\subsection{Octagonal cylinders and stack of rectangles}

		We now consider 42531 and 53142-avoiding permutations that contain exactly one 4321-pattern. If $w$ is such a permutation, then $G(w)$ is an octagonal cylinder in higher dimension.

		\begin{ex}
			Below is the graph $G(4, 3, 2, 1, 5, 6, 7, 11, 10, 8, 9)$.
			
			\begin{center}
				\includegraphics[width=0.50\textwidth]{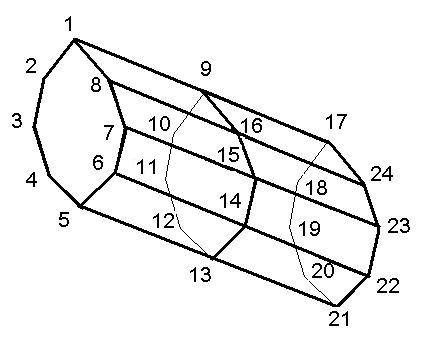}
			\end{center}

			Each number represents a commutation class. The following table shows an element from each commutation class.

		\begin{tabular}[t]{|c|c|c|c|}
    	\hline
    	Vertex & Commutation Class & Vertex & Commutation Class\\\hline
    	1 & (3, 2, 1, 3, 2, 3, 10, 9, 10, 8, 9) & 2 & (3, 2, 1, 2, 3, 2, 10, 9, 10, 8, 9)\\
    	3 & (3, 1, 2, 1, 3, 2, 10, 9, 10, 8, 9) & 4 & (1, 2, 3, 2, 1, 2, 10, 9, 10, 8, 9)\\
    	5 & (1, 2, 3, 1, 2, 1, 10, 9, 10, 8, 9) & 6 & (2, 1, 2, 3, 2, 1, 10, 9, 10, 8, 9)\\
    	7 & (2, 3, 1, 3, 1, 3, 10, 9, 10, 8, 9) & 8 & (2, 3, 2, 1, 2, 3, 10, 9, 10, 8, 9)\\
    	9 & (3, 2, 1, 3, 2, 3, 10, 9, 10, 8, 9) & 10 & (3, 2, 1, 2, 3, 2, 9, 10, 9, 8, 9)\\
    	11 & (3, 1, 2, 1, 3, 2, 9, 10, 9, 8, 9) & 12 & (1, 2, 3, 2, 1, 2, 9, 10, 9, 8, 9)\\
    	13& (1, 2, 3, 1, 2, 1, 9, 10, 9, 8, 9) & 14 & (2, 1, 2, 3, 2, 1, 9, 10, 9, 8, 9)\\
    	15 & (2, 3, 1, 3, 1, 3, 9, 10, 9, 8, 9) & 16 & (2, 3, 2, 1, 2, 3, 9, 10, 9, 8, 9)\\
    	17 & (3, 2, 1, 3, 2, 3, 9, 10, 8, 9, 8) & 18 & (3, 2, 1, 2, 3, 2, 9, 10, 8, 9, 8)\\
    	19 & (3, 1, 2, 1, 3, 2, 9, 10, 8, 9, 8) & 20 & (1, 2, 3, 2, 1, 2, 9, 10, 8, 9, 8)\\
    	21 & (1, 2, 3, 1, 2, 1, 9, 10, 8, 9, 8) & 22 & (2, 1, 2, 3, 2, 1, 9, 10, 8, 9, 8)\\
    	23 & (2, 3, 1, 3, 1, 3, 9, 10, 8, 9, 8) & 24 & (2, 3, 2, 1, 2, 3, 9, 10, 8, 9, 8)\\
    	\hline
		\end{tabular}
		\end{ex}

		A natural follow-up question is what does $G(w)$ look like if $w$ contains multiple 4321-patterns? However the intersections of 8-cycles are too complicated to visualize.
		
		We now consider 4321-avoiding permutations. Since there are no 8-cycles, $G(w)$ is composed of only 4-cycles, and thus we can obtain a stack of rectangles.

	\begin{ex}
    The graph $G(5, 3, 1, 6, 4, 2)$ is a combination of a cube and a square as shown below.

    \centerline{
    \xymatrix{
    & 1 \ar[d]\ar[dr]&\\
    & 3 \ar[dl]\ar[d]\ar[dr] & 2\ar[d]\\
    7 \ar[d] \ar[dr] & 5 \ar[dl]\ar[dr] & 4\ar[dl]\ar[d] \\
    9\ar[dr] & 8 \ar[d] & 6 \ar[dl] \\
    & 10 & \\
    }
    }

		The following table lists a reduced decomposition from each commutation class.

		\begin{tabular}[t]{|c|c|c|c|}
    \hline
    Vertex & Commutation Class & Vertex & Commutation Class\\\hline
    1 & (4, 3, 2, 3, 4, 1, 2, 5, 4, 5) & 2 & (4, 2, 3, 2, 4, 1, 2, 5, 4, 5)\\
    3 & (2, 4, 3, 4, 2, 1, 2, 5, 4, 5) & 4 & (2, 4, 3, 4, 2, 1, 2, 4, 5, 4)\\
    5 & (2, 4, 3, 4, 1, 2, 1, 5, 4, 5) & 6 & (2, 4, 3, 4, 1, 2, 1, 4, 5, 4)\\
    7 & (2, 3, 4, 3, 2, 1, 2, 5, 4, 5) & 8 & (2, 3, 4, 3, 2, 1, 2, 4, 5, 4)\\
    9 & (2, 3, 4, 3, 1, 2, 1, 5, 4, 5) & 10 &(2, 3, 4, 3, 1, 2, 1, 4, 5, 4)\\
    \hline
		\end{tabular}
	\end{ex}




\section{Subnetworks}

In this section, we study the enumeration of subnetworks. We first investigate the relationship between subnetworks and commutation classes.

	\begin{lemma} \label{equivSub}
    Suppose a permutation $w$ contains a pattern $p$, and $X$ is a commutation class of reduced decompositions of $p$. Then the number of $X$-subnetworks is the same for each reduced decomposition in a commutation class of $w$.
	\end{lemma}

	\begin{proof}
	
	Let $\rho$ be a reduced decomposition in a commutation class of $w$. Suppose a short braid move changes $s_is_j$ into $s_js_i$ where $s_i$ transposes $(a, b)$ and $s_j$ transposes $(c, d)$. Let $\rho'$ be the image of $\rho$ after this move. Since $|i-j|\ge 2$, $a, b, c, d$ are all distinct. If $p\in S_k$, then we pick $k$ elements from $\{1, 2, \ldots, n\}$ whose reduced decomposition is $t\in X$ when we apply $\rho$ to $w$, and $t'$ when we apply $\rho'$ to $w$. If these $k$ elements contain all of $\{a, b, c, d\}$, then $t$ transposes $(a, b)$ and $(c, d)$ in two consecutive moves. We can interchange these two moves to obtain $t'$. Thus $t'$ is obtained by applying a short braid move to $t$, which means $t'\in X$. If not all of $\{a, b, c, d\}$ belong to these $k$ elements, then $t'=t$, which is in $X$.
	\end{proof}

	We now present two interesting applications of Lemma \ref{equivSub}.

	\begin{defn}
    A permutation $w$ is \textit{p-friendly} if any $321$-pattern of $w$ is contained in exactly $k$ $p$-patterns for a constant $k$.
	\end{defn}

	\begin{ex}
		The longest word $w_0\in S_n$ is $p$-friendly if $p$ is the longest word of $S_m$. In this case each triple is contained in $\binom{n-3}{m-3}$ $p$-patterns.
	\end{ex}

	\begin{prop}
    Suppose $p$ contains exactly one $321$-pattern, and $w$ is $p$-friendly with every $321$-pattern contained in exactly $k$ $p$-patterns. Let $X$ be the commutation class of the highest rank in $P(p)$. Let $c$ denote the sum of indices of the commutation class of the lowest rank in $P(w)$. Then the number of $X$-subnetworks in a reduced decomposition $i_1, i_2, \ldots, i_l$ of $w$ is equal to $(k\sum_{j=1}^{l}i_j)-c$.
	\end{prop}

	\begin{proof}
     By Lemma $\ref{equivSub}$, the number of $X$-subnetworks is the same in any commutation classes of $w$. Now let's consider the effect of a long braid move. If a long braid move changes $i+1, i, i+1$ to $i, i+1, i$, then the triple it acted on must be contained in exactly $k$ $p$-patterns. Since $p$ contains exactly one $321$-pattern, there are exactly two commutation classes of the reduced decompositions of $p$. Thus number of $X$-subnetworks decreases by $k$. On the other hand, the sum of indices decreases by exactly $1$. Therefore, the number of subnetworks changes at a rate exactly $k$ times the sum of indices. Since the commutation class of the lowest rank in $P(w)$ (whose sum of indices is equal to $c$) has no $X$-subnetworks, the number of $X$-subnetworks in a reduced decomposition $i_1, i_2, \ldots, i_l$ of $w$ is exactly $(k\sum_{j=1}^{l}i_j)-c$.
	\end{proof}

We now discuss the case when $p$ is the longest word of $S_4$.

	\begin{prop} \label{subEnum}
    Let $\rho=(i_1, i_2, \ldots ,i_{l})$ be a reduced decomposition of the long word $w_0=(n, n-1, \ldots, 1)$, where $l=\binom{n}{2}$. Let $X=\{123212, 321232, 212321, 232123\}$. Then the number of $X$-subnetworks induced by $\rho$ is equal to $$\displaystyle\sum_{j=1}^{l}(i_j-1)(n-i_j-1)-2\dbinom n 4.$$
	\end{prop}

	\begin{proof}
    First notice that this formula is true for the reduced decomposition $(n-1)(n-2, n-1)(n-3, n-2, n-1)\cdots(k, k+1,\ldots ,n-1)\cdots (1,\ldots, n-2, n-1)$. (We basically move $n$ to the first position, then move $n-1$ to the second position, and so on.) The formula yields exactly zero.

    Now by Lemma $\ref{equivSub}$, the number of $X$-subnetworks is invariant under a short braid move. Thus we only need to consider the change under a long braid move. Notice that no words in $X$ contain $121$ or $323$, and furthermore only words in $X$ contain $212$ or $232$. Suppose a long braid move $i+1, i, i+1$ affects the triple $(a, b, c)$ where $a<b<c$. Before we apply this move, a quadruple $(a, b, c, d)$ is an $X$-subnetwork if and only if $d$ is to the right of $c$. After this move, a quadruple $(a, b, c, d)$ is an $X$-subnetwork if and only if $d$ is to the left of $c$. Thus the net change in the number of $X$-subnetworks is $2i-n+1$, which is exactly what the formula would give.
	\end{proof}

	\begin{remark}
    In fact, all four elements of $X$ occur with equal probability over all reduced decompositions of $w_0$. We conclude our discussion with a generalization of this interesting phenomenon.
	\end{remark}

	\begin{defn}
    Let $w\in S_n$. The \textit{complement} of a reduced decomposition $(r_1, r_2, \ldots, r_l)$ of $w$ is $(n-r_1, n-r_2, \ldots, n-r_l)$. The \textit{reverse} of $(r_1, r_2, \ldots, r_l)$ is $(r_l, r_{l-1}, \ldots, r_1)$.
	\end{defn}
	
Note that the complement and the reverse may not exist for an arbitrary reduced decomposition.


	\begin{prop}
    Let $w\in S_n$, $p\in S_m$, and let $x$ and $x'$ be two reduced decompositions of $p$.
    \begin{enumerate}
        \item  If $x$ and $x'$ are reverses, then the total number of $x$-subnetworks is equal to the total number of $x'$-subnetworks over all reduced decompositions of $w$.
        \item If $w$ is the longest word of $S_n$, and $p$ is the longest word of $S_m$, then $x$ must have a complement $x'$. Furthermore, the total number of $x$-subnetworks is equal to the total number of $x'$-subnetworks over all reduced decompositions of $w$.
    \end{enumerate}
	\end{prop}

	\begin{proof}
    Let $\rho$ denote a reduced decomposition of $w$. Then we claim that there exists an $\rho'$ such that an $x$-subnetwork of $\rho$ can be mapped bijectively to an $x'$-subnetwork of $\rho'$. Since we consider all reduced decompositions of $w$, the total number of $x$-subnetworks is equal to the total number of $x'$-subnetworks.

    \begin{enumerate}

        \item Given any reduced decomposition $\rho$ of $w$, let $\rho'$ denote the reverse of $\rho$. Then for any pair of transpositions $\boldsymbol{i}$ and $\boldsymbol{j}$, we have $\boldsymbol{i}$ occurs before $\boldsymbol{j}$ in $\rho$ implies that $\boldsymbol{i}$ occurs after $\boldsymbol{j}$ in $\rho'$, and vice versa. Therefore, every $x$-subnetwork of $\rho$ is mapped to an $x'$-subnetwork of $\rho'$.

        \item If $p$ is the longest word of $S_m$, then the complement of $x$ is obtained by applying Reiner's cyclic shifting trick (see $\cite{Reiner}$) $m$ times. Thus every reduced decomposition of the longest word has a complement. Given a reduced reduced decomposition $\rho$ of $w$. Let $\rho'$ denote the complement of $\rho$. Suppose $(a_1, a_2, \ldots, a_m)$ is an $x$-subnetwork, then $(n-a_m, n-a_{m-1}, \ldots, n-a_1)$ is an $x'$-subnetwork.

    \end{enumerate}
	\end{proof}



\section{Remarks and questions}

	In this paper, we studied three aspects of commutation classes: the size of $G(w)$, the structure of $G(w)$, and the enumeration of subnetworks. Our concluding remarks are also divided into three sections.
	
	\subsection{The size of $G(w)$}
	
		We first examine our bound in Theorem \ref{bounding} for four types of permutations:
		
		\begin{itemize}
  			\item \textit{321 avoiding}: The lower bound yields 1, which is exactly the number of vertices of $G$.
  			\item \textit{Y=1}: In this case, $G(w)$ is a path. By Corollary $\ref{pathGraph}$, there are $N_{321}(w)+1$ vertices, which favors the lower bound.
  			\item \textit{$w_0=n, n-1, \ldots, 2, 1$}: The exact number is not known, but Knuth showed in $\cite{Knuth}$ that $Y(w_0)>n^2/6-O(n)$. Since $l=\binom {n} {2}$, both bounds are decent.
  			\item \textit{Freely Braided}: By Corollary $\ref{freeBraid}$, $G=2^Y$, which usually favors the lower bound unless $Y$ is close to $l$.
		\end{itemize}
				
		We note here that Green and Losonczy \cite{Green} proved that the $|G(w)|\le 2^{N_{321}}$. Since the number of 321-patterns increases dramatically as the number of inversions increases, this bound is more accurate for permutations with a small number of inversions, while the upper bound in Theorem \ref{bounding} is more accurate for permutation with a large number of inversions. For example, the permutation $w_0$ has $\binom n 3$ patterns of 321, but it only has $\binom n 2$ inversions. Can we refine the upper bound by combining these two bounds? Counting the number of Elnitsky Polygons might be a useful tool.
		
		In the proof of Theorem \ref{boundSum}, we bounded the number of commutation classes by the Catalan number. The representative elements are in some sense variations of the bidirectional ballot sequence studied by Zhao \cite{Zhao}. Can we refine the bound in Theorem \ref{boundSum} by studying such sequences? Does the combination of Theorem \ref{bounding} and Theorem \ref{boundSum} yield tighter bounds?
		
		Our method in the proof of the upper bound of Theorem \ref{boundSum} is inspired by the author's recent work \cite{Meng} on the higher Bruhat order. The higher Bruhat order is a generalization of the symmetric group. Is there an analogue of the higher Bruhat order for Coxeter groups of type B? Such a generalization might help us to refine the bound for type B in Section \ref{upperbound}.

	\subsection{The structure of $G(w)$}
	
	In a recent paper (\cite{Meng} Theorem 4.2), the author showed that $G(w)$ is isomorphic to an induced subgraph of $G(w_0)$ for all $w$. This implies that analyzing the structure of special $w$ is crucial for understanding the overall structure of $G(w_0)$. Ziegler (\cite{Ziegler} Theorem 5.2) combined with Felsner and Weil (\cite{FelsnerWeil} Theorem 1) showed that $G(w_0)$ is homotopy equivalent to an $(n-4)$-sphere. What is the topological structure of $G(w)$ for arbitrary $w$?
	
	In Section 4, we studied the cycle structure of $G(w)$ using pattern avoidance. However, we are unable to describe $G(w)$ if $w$ contains multiple 4321-patterns. One possible approach to this complicated problem is to start with the case when the 4321-pattern do not intersect with each other. Does $G(w)$ have a nice geometric structure if $w$ contains no intersecting 4321-patterns?
	
	Now let's talk about involutions. Following the work of Richardson and Springer \cite{Richardson}, there has been a surge of interest in the strong Bruhat order on the involutions (e.g. \cite{Hultman1, Hultman2, Hultman}.) However, the weak Bruhat order on involutions has never been studied before. This is because the poset induced by the weak Bruhat order on involutions is not a graded poset, and reduced decomposition cannot be defined for a non-graded poset. Nevertheless, Incitti's pictorial classification of cover relations in the strong Bruhat order on involutions in \cite{Incitti} can be easily translated to the weak Bruhat order. Can we find an analogue of $G(w)$ for involutions that avoid certain patterns?

	
	\subsection{Subnetworks}
		
		The subnetwork is a useful tool to study commutation classes. In Section \ref{poset}, we constructed commutation classes using 212-subnetworks. The proof of Theorem \ref{recPerm} essentially used the fact that a subnetwork appears as a substring only if the permutation is vexillary by Tenner's result in \cite{Tenner}. Since a subnetwork records the relative positions of a subset of $\{1, 2, \ldots, n\}$, it can be thought of as a permutation pattern for reduced decompositions. Therefore it is natural for us to study enumerations of subnetworks.
		
		Unlike permutation patterns, however, the study of subnetworks is a relatively new subject.
		
		The concept of subnetworks was first introduced by Warrington \cite{Warrington} in November 2009 (when his preprint first appeared on the arXiv). Since then, Angel and Holroyd \cite{Angel} studied the enumeration of subnetworks for the first time. This subject is still wide open.
		
		One interesting topic is subnetwork-avoidance. We say that a reduced decomposition is \textit{$X$-avoiding} if it does not induce any $X$-subnetworks. In Proposition \ref{subEnum}, we found the exact number of $X$-subnetworks induced by the reduced decomposition of the long word $w_0$ for $X=\{123212, 321232, 212321, 232123\}$, but our result does not immediately give the number of reduced decompositions that avoids $X$. Warrington \cite{Warrington1} has generated the following data for the number of $X$-avoiding reduced decompositions starting from $n=3$:
		
		$$2, 12, 328, 54520, 68641152.$$
		
		\paragraph*{Question} Let $X=\{123212, 321232, 212321, 232123\}$, and $w_0=n, n-1, \ldots, 1$. How many reduced decompositions of $w_0$ are $X$-avoiding? How many commutation classes are $X$-avoiding?


\section*{Acknowledgements}
	
Particular thanks are due to Richard Stanley for his suggestion to study reduced decompositions and for his continued guidance.
	
Part of this research was carried out at the University of Minnesota Duluth under the supervision of Joseph Gallian, with the financial support of the National Science Foundation and the Department of Defense (grant number DMS 0754106), the National Security Agency (grant number H98230-06-1-0013), and the MIT Department of Mathematics. The author would like to thank Joseph Gallian for his encouragement and support.

In addition, the author would like to thank Vic Reiner for pointing out the works of Felsner and Ziegler on the higher Bruhat order, and Francesco Brenti for reading this paper and making valuable suggestions.


\begin{thebibliography}{30}

\bibitem{Angel} O. Angel and A. Holroyd. Random subnetworks of random sorting networks. Electr. J. Combin., 17:N23, 2010.
\bibitem{Bedard} R. Bedard. On commutation classes of reduced words in Weyl groups. Europ. J. Combinatorics, 20:483-505, 1999.
\bibitem{Bjorner} A. Bjorner and F. Brenti. Combinatorics of Coxeter groups, volume 231.
Springer, New York, 2005.
\bibitem{Elnitsky} S. Elnitsky. Rhombic tilings of polygons and classes of reduced words in
Coxeter groups. J. Combin. Theory, Ser. A, 77:193-221, 1997.
\bibitem{Felsner} S. Felsner and P. Valtr. Coding and counting arrangements of pseudolines.
preprint.
\bibitem{FelsnerWeil} S. Felsner and H. Weil. A theorem on the higher Bruhat order. Discrete $\&$
Computational Geometry, 23:121-127, 2000.
\bibitem{Green} R. Green and G. Losonczy. Freely braided elements in Coxeter groups.
Annals of Combinatorics, 6:337-348, 2002.
\bibitem{Hultman} A. Hultman. Fixed points of involutive automorphisms of the Bruhat order.
Adv. Math., 195:283-296, 2005.
\bibitem{Hultman1} A. Hultman. The combinatorics of twisted involutions in Coxeter groups.
Trans. Amer. Math. Soc., 359:2787-2798, 2007.
\bibitem{Hultman2} A. Hultman and K. Vorwerk. Pattern avoidance and Boolean elements in
the Bruhat order on involutions. J. Algebraic Combin., 30:87-102, 2009.
\bibitem{Incitti} F. Incitti. The Bruhat order on the involutions of the symmetric group. J.
Algebraic Combin., 20:243-261, 2004.
\bibitem{Knuth} D. Knuth. Axioms and Hulls. Springer-Verlag, 1992.
\bibitem{Meng} D. Meng. Boolean elements in the bruhat order on twisted involutions.
preprint.
\bibitem{Meng1} D. Meng. Reduced decompositions and permutation patterns generalized
to the higher Bruhat order. preprint.
\bibitem{Reiner} V. Reiner. Note on the expected number of Yang-Baxter moves applicable
to reduced decompositions. Europ. J. Combin., 26:1019-1021, 2005.

\bibitem{Richardson} R.W. Richardson and T.A. Springer. The Bruhat order on symmetric va-
rieties. Geom. Dedicata., 35:389-436, 1990.
\bibitem{Shapiro} B. Shapiro, M. Shapiro, and A. Vainshtein. Connected components in the
intersection of two open opposite Schubert cells in SLn(R)=B. Intern.
Math. Res. Notices., 10:469-493, 1997.
\bibitem{Stanley} R.P. Stanley. Permutations. notes for the 2010 AMS Colloquium Lectures.
\bibitem{Tenner} B.E. Tenner. Reduced decompositions and permutation patterns. J. Alge-
braic Combin., 24:263-284, 2006.
\bibitem{Warrington1} G. Warrington. private communication.
\bibitem{Warrington} G.Warrington. A combinatorial version of Sylvestor's Four-Point Theorem.
Advances in Applied Mathematics, 45:390-394, 2010.
\bibitem{Zhao} Y. Zhao. Constructing MSTD sets using bidirectional ballot sequences. J.
Number Theory, 130:1212-1220, 2010.
\bibitem{Ziegler} G. Ziegler. Higher Bruhat orders and cyclic hyperplane arrangements.
Topology, 32:259-279, 1993.

\end{thebibliography}
\end{document}